\documentclass[11pt]{amsart}
\usepackage[top=4cm,bottom=4cm,left=4cm,right=4cm]{geometry}               
\usepackage{amsmath,amscd,amssymb,amsthm}
\usepackage[english]{babel}
\usepackage{mathtools}
\usepackage{enumerate}
\usepackage{setspace}
\usepackage{mathrsfs}
\usepackage{xcolor}
\usepackage[numbers]{natbib}
\usepackage[hidelinks]{hyperref}

\DeclareMathOperator{\Hom}{Hom}
\DeclareMathOperator{\Gal}{Gal}
\DeclareMathOperator{\St}{St}
\DeclareMathOperator{\Chr}{char}

\def\vF{\mathbb{F}}
\def\vN{\mathbb{N}}

\def\cO{\mathcal{O}}
\def\cS{\mathcal{S}}

\newtheorem{theorem}{Theorem}
\newtheorem{lemma}[theorem]{Lemma}
\newtheorem{corollary}[theorem]{Corollary}
\newtheorem{proposition}[theorem]{Proposition}
\newtheorem{conjecture}[theorem]{Conjecture}
\theoremstyle{definition}
\newtheorem{definition}[theorem]{Definition}

\newtheorem{remark}[theorem]{Remark}

\makeindex

\title[On the selection of polynomials for the DLP algorithm]{On the selection of polynomials for the DLP quasi-polynomial time algorithm in small characteristic}
\author[Giacomo Micheli]{Giacomo Micheli}

\begin{document}

\maketitle

\begin{abstract}
In this paper we characterize the set of polynomials $f\in\mathbb F_q[X]$ satisfying the following property: there exists a positive integer $d$ such that for any positive integer $\ell$ less or equal than the degree of $f$, there exists $t_0$ in $\mathbb F_{q^d}$ such that the polynomial $f-t_0$ has an irreducible factor of degree $\ell$ over $\mathbb F_{q^d}[X]$. This result is then used to progress in the last step which is needed to remove the heuristic from one of the quasi-polynomial time algorithms for discrete logarithm problems (DLP) in small characteristic. Our characterization allows a construction of polynomials satisfying the wanted property. The method is general and can be used to tackle similar problems which involve factorization patterns of polynomials over finite fields.
\end{abstract}


\section{Introduction}

For a long time the discrete logarithm problem (DLP) over finite fields has been one of the most important primitives used for cryptographic protocols. The major breakthrough in recent years concerning DLPs in small characteristic consists of the heuristic quasi-polynomial time algorithms given in \cite{Barbulescu2014,granger2015discrete} (see also \cite{EPFL-CONF-215123,joux2013new} for their origins).

In this paper we focus on the algorithm in \cite{granger2015discrete} which only relies on the field representation heuristic (see \cite[p.2]{granger2015discrete}). In fact, if that can be proved, this would show that DLP in small characteristic can indeed be solved in quasi-polynomial time. Our results characterize a class of polynomials which seem to be particularly suitable for performing the quasi-polynomial time DLP-algorithm described in \cite{granger2015discrete} and show that if one wants to select polynomials satisfying the wanted property, these have to be chosen in this class (see Theorem \ref{monodromy}). 

Our constructions involve some Galois theory over function fields, group theory and Chebotarev density theorem. Let us start with the motivating conjecture, which has to be proved in order to remove the remaining heuristic from the algorithm in \cite{granger2015discrete}.

\begin{conjecture}
For any finite field $\vF_q$ and any fixed positive integer $\ell\leq q+2$, there exists an integer $d=O(\log(q))$ and $h_1,h_2\in\vF_{q^d}[X]$ coprime of degree at most $2$ such that $h_1X^q+h_2$ has an irreducible factor of degree $\ell$.
\end{conjecture}

If this conjecture is true, then DLP in small characteristic can be solved in non-heuristic quasi-polynomial time as described in the algorithm presented in \cite{granger2015discrete}. 

Such kind of requirement also appeared in \cite[Section 5]{Barbulescu2014}  where it is observed that the choice $h_1=1$ and $h_2=X^2-t_0$ (for some well chosen $t_0\in \vF_{q^d}$) seems to always satisfy the requirements  in odd characteristic and for $d=2$.
This motivates us to formulate the following stronger
\begin{conjecture}\label{strongerconj}
Let $\mathbb F_q$ be a finite field of odd characteristic.
There exists an integer  $d=O(\log(q))$ and $h_1,h_2\in\vF_{q^d}[X]$ coprime of degree at most $2$ such that, for any positive integer $\ell \leq \deg(h_1)+q$  there exists $t_0\in \vF_{q^d}$  such that $h_1X^q+h_2-t_0$ has an irreducible factor of degree $\ell$.
\end{conjecture}

A polynomial satisfying Conjecture \ref{strongerconj} will allow to build extensions with the correct representation and of desired degree. 

Both these conjectures seem to be very hard. In this paper we make a step forward by showing a relaxed version of the stronger conjecture: in fact, we will fit the conjecture above in a general framework and will show a characterization of polynomials satisfying a weaker property than the one described in Conjecture \ref{strongerconj}. In particular we will be able to prove the following
\begin{theorem}\label{theoremoddh1h2}
Let $\mathbb F_q$ be a finite field of odd characteristic.
There exists an integer  $d\in \vN$ and $h_1,h_2\in\vF_{q^d}[X]$ coprime of degree at most $2$ such that, for any positive integer $\ell \leq \deg(h_1)+q$  there exists $t_0\in \vF_{q^d}$  such that $h_1X^q+h_2-t_0$ has an irreducible factor of degree $\ell$. Moreover, such polynomials can be constructed explicitly.
\end{theorem}

More in general, we characterize completely (in any characteristic) polynomials $f\in \vF_q[X]$ having the property that there exists a $d\in \vN$ such that for any $\ell\leq \deg(f)$, there exists $t_0\in \vF_{q^d}$ such that $f-t_0$ has an irreducible factor of degree $\ell$ in $\vF_{q^d}[X]$.

On the theoretical side, our result shows the existence of such $d$ for a certain class of polynomials, which is the first step in the attempt of giving an explicit bound. In practice, our methods are constructive and they allow to build new families of polynomials (see for example the constructions in subsection \ref{xqjjx}) which always satisfy the wanted requirements. Even though we can show the existence of such $d$ for these families of polynomials, the wrinkle is that the required $d$ might in principle be large (but in practice, if one follows our recipe, this seems to be never the case).  

In a nutshell, what we will do in this paper is to solve the geometric part of the problem connected with the two conjectures above and what remains to do to completely remove the heuristic is to  give an explicit logarithmic bound for $d$ for at least one polynomial in our families.

The key idea of the method is the following. We look at the problem in a function field theoretical framework, explaining that the factorization conditions can be translated into group theoretical properties of the Galois closure of a certain extension $L:K$ of global function fields. Then, we use the rigidity of group theory to determine the Galois group that can occur for the polynomials we are interested in. Finally, Chebotarev Density Theorem for global function fields will ensure that, for any fixed element $\gamma$ in the Galois group, there exists an unramified place $P$ of $K$ for which the cycle decomposition of $\gamma$ (when you look at its action on a certain set of homomorphisms) appears exactly as the splitting of $P$ in $L$. 

The paper is structured as follows. In Section \ref{sec:dunigalois} we recap the basic tools we need from algebraic number theory and group theory. In Section \ref{sec:characterizationofuni} we characterize the monodromy groups of the class of polynomials we are interested in. In Section \ref{sec:Xq+X2-t} we specialize to the polynomial $X^q+X^2$ and compute its monodromy group, showing that for odd $q$, it is indeed the full symmetric group. In Section $5$ we show other examples of polynomials of the wanted form that have symmetric monodromy group.

\subsection{Notation}
For the entire paper $p$ is a prime (even or odd) and $q=p^a$ for some positive integer $a$. Let $k:=\vF_q$ be the finite field of order $q$. Let $f\in k[X]\setminus k[X^p]$. Let $M_f$ be the splitting field of $f-t$ over $k(t)$, which is a separable extension of $k(t)$. Let $\tilde k$ be the field of constants of $M_f$ i.e. the integral closure of $k$ in $M_f$. Let $A_f=\Gal(M_f:k(t))$ be the \emph{arithmetic monodromy group} of $f$ and  $G_f=\Gal(M_f:\tilde{k}(t))\trianglelefteq A_f$ be the \emph{geometric monodromy group} of $f$.
Let $\mathcal S_n$ be the symmetric group of degree $n$. Notice that if $F_1,F_2$ are subfields of a larger field $F$, we denote by $F_1F_2$ the the compositum of $F_1$ and $F_2$.
Let $G$ be a group acting on a set $Y$. For any $y\in Y$ we denote by $\St_G(y)$ the stabilizer of $y$ in $G$.

\section{$d$-universal polynomials and Galois Theory over Function Fields}\label{sec:dunigalois}
In this section we define the notion of universal polynomial and state the basic results from global function field theory we will be using in the rest of the paper.
\begin{definition}
Let $f\in \vF_q[X]$. We say that $f$ is $d$-\emph{universal} for some positive integer $d$ if for any positive integer $\ell\leq\deg(f)$, there exists $t_0$ in $\vF_{q^d}$ such that $f(X)-t_0$ has an irreducible factor of degree $\ell$. We say that $f$ is \emph{universal} if it is $d$-universal for some $d$.
\end{definition}

\begin{remark}
In this notation, in \cite{Barbulescu2014} it is suggested that $X^q+X^2$ is $2$-universal for any odd $q$ (see section \emph{Finding appropriate $h_0$, $h_1$} of \cite{Barbulescu2014}).
\end{remark}

In what follows we will use notation and terminology of \cite{stichtenoth2009algebraic}. First, we need a classical result from algebraic number theory, which will be used to transfer splitting conditions of places into into group theoretical properties of a certain Galois group.
\begin{theorem}\label{orbits}
Let $L:K$ be a finite separable extension of global function fields and let $M$ be its Galois closure with Galois group $G$. Let $P$ be a place of $K$ and $\mathcal Q$ be the set of places of $L$ lying over $P$.
Let $R$ be a place of $M$ lying over $P$. There is a natural bijection between $\mathcal Q$ and the set of orbits of $H=\Hom_K(L,M)$ under the action of the decomposition group $D(R|P)=\{g\in G\,|\, g(R)=R\}$.
In addition, let $Q\in \mathcal Q$ and let $H_Q$ be the orbit corresponding to $Q$. Then $|H_Q|=e(Q|P)f(Q|P)$ where $e(Q|P)$ and $f(Q|P)$ are ramification index and relative degree respectively.
\end{theorem}
A proof of Theorem \ref{orbits} can be found for example in \cite{guralnick2007exceptional}.
For a finite Galois extension of function fields $M:K$ with Galois group $G$, let $P$ be a degree $1$ place of $K$ and $R$ be a place of $M$ lying over $P$. Let $\phi$ be the topological generator of $\Gal(\overline{k}:k)$ defined by $y\mapsto y^q$.
Let $k_R$ be the residue field at $R$ and let $\phi_R$ be the image of $\phi$ in $\Gal(k_R:k)$. If $(R,M:K)$ is the set of elements in $D(R|P)$ mapping to $\phi_R$, we denote by $(P,M:K)$ the set $\{gxg^{-1}: \; g\in G, x\in (R,M:K)\}$.

We are now ready to state the other fundamental tool, which can easily be adapted from \cite{kosters2014short}.
\begin{theorem}[Chebotarev Density Theorem]\label{handlechgeneral}
Let $M:K$ be a finite Galois extension of function fields over a finite field $k$ of cardinality $q$ and let $\tilde k$ be the constant field of $M$. Let $A=\Gal(M:K)$ and $G=\Gal(M:\tilde k K)$. Let $\gamma\in A$ such that $\gamma$ acts as $u\mapsto u^q$ when restricted to $\tilde k$. Let $g\in G\gamma$, $\Gamma$ be the conjugacy class of $g$ and let $S_K$ be the set of places in $K$ which are unramified in $M$. Then we have
\[|\{P\in S_K| \deg_k(P)=1, (P,M:K)=\Gamma\}|=\frac{|\Gamma|}{|G|}q+2\frac{|\Gamma|}{|G|}
\mathfrak{g}_M q^{1/2}\]
where  $\mathfrak{g}_M$ is the genus of $M$.
\end{theorem}

Theorem \ref{handlechgeneral} combined with Theorem \ref{orbits} is used to push group theoretical information to splitting statistics: the key fact is that the number of elements in the Galois group with a certain cycle decomposition (and in the correct coset of the geometric Galois group) determines the statistics of the unramified places that split according to the given cycle decomposition. 
Let us give an example that clarifies the procedure for the class of extensions we are interested in. Let $f$ be a polynomial of degree $n\geq 6$ in $\vF_q[X]$ and consider the polynomial $f-t\in \vF_q(t)[X]$. Set $L=\vF_q(x)=\vF_q(t)[x]/(f(x)-t)$, $K=\vF_q(t)$, and $M_f$ as in the notation section (i.e. the Galois closure of $L:K$). Observe first that $\Hom_K(L,M_f)$ is in natural correspondence with the roots of $f(x)-t$ in $M_f$, and in turn the action of $A_f$ on $\Hom_K(L,M_f)$ is equivalent to the action of $A_f$ on the roots of $f-t$. 
Now, suppose for example we want to know an estimate for the number of $t_0$'s in $\vF_q$ such that $f-t_0$ splits into two degree $2$ irreducible factors and a degree $n-4$ irreducible factor in $\vF_q[x]$. Let now $\gamma\in A$ be the Frobenius (i.e. $x\mapsto x^q$) for the field $k_f$.
Take now the coset $G_f\gamma$ and take the set $Z$ of all elements in $G_f\gamma$ with disjoint  cycle decomposition 
\[(-,-)(-,-)\underbrace{(-,\dots,-)}_{n-4}\] 
when you look at their action on the roots of $f-t$.
Notice that $Z\subseteq G_f\gamma$ is a union of $A_f$-conjugacy classes, as $A_f/G_f$ is cyclic. Applying Chebotarev for each of the conjugacy classes and adding the estimates together gives that the number of $t_0\in \vF_q$ such that $f-t_0$ has the wanted factorization pattern is then $q|Z|/|G_f| +O(\sqrt{q})$, where the implied constant can be chosen independent of $q$.

In what follows we will only need the following special version of Chebotarev density theorem,  which can be also derived from \cite{rosen2013number}.

\begin{theorem}[Chebotarev Density Theorem with trivial constant field extension]\label{handlecheb}
Let $M:K$ be a finite Galois extension of function fields over a finite field $k$ of cardinality $q$. Let $G=\Gal(M:K)$ and assume that the field of constants of $M$ is exactly $k$. 
Let $\Gamma$ be a conjugacy class of $G$ and let $S_K$ be the set of places in $K$ which are unramified in $M$. Then we have
\[|\{P\in S_K| \deg_k(P)=1, (P,M:K)=\Gamma\}|=\frac{|\Gamma|}{|G|}q+O\left(q^{1/2}\right).\] 
\end{theorem}

The following easy lemma simplifies some of the proofs of the results in this paper.
\begin{lemma}\label{compositumfinitefield}
Let $f$ be a separable polynomial, let $k'$ be an extension of $k$, and $\tilde k':=k'\cap \tilde k$. Then
\[\Gal(k'M_f:k'(t))\cong\Gal(M_f:\tilde k'(t)).\]
\end{lemma}
\begin{proof}
First we observe that if $F_1=M_f$ and $F_2=k'(t)$, then $F_1\cap F_2=\tilde k'(t)$. In addition, we know the Galois group of the compositum:
\[\Gal(F_1F_2: F_1\cap F_2)=\Gal(k'M_f:\tilde k'(t))\cong\Gal(M_f:\tilde k'(t))\times\Gal(k'(t):\tilde k'(t))\]
where the isomorphism is defined by the restriction map to $M_f$ and $k'(t)$. It follows easily that 
\[\Gal(k'M_f:k'(t))\cong\Gal(M_f:\tilde k'(t)).\]
\end{proof}
 
\subsection{Short Group Theory Interlude}

\begin{definition}
Let $X$ be a finite set and $G$ be a finite group. An action of $G$ on $X$ is said to be \emph{non-primitive} if there exists an integer $\ell\in \{2,\dots ,|G|-1\}$ and a partition of $X$ into $X_1,\dots X_\ell$ such that for any $i\in \{1,\dots, \ell\}$ and any $g\in G$ we have $g(X_i)=X_{i_g}$ for some $i_g\in \{1,\dots,\ell\}$.
An action is said to be \emph{primitive} if is not non-primitive. 
\end{definition}

Roughly, the above definition states that an action of a group $G$ on a set $X$ is primitive if it does not preserve any non-trivial partition of $X$. We will also need the following group theory lemma, of which we include the proof for completeness.

\begin{lemma}\label{primitivegroup}
Let  $G$ be a subgroup of $\cS_n$ acting on $U=\{1,\dots,n\}$.
Suppose that $G$ acts transitively on $U$ and it contains a cycle of prime order $r$ with $r>n/2$. Then $G$ acts primitively on $U$.
\end{lemma}
\begin{proof}
Let $X_1\sqcup X_2\sqcup \dots \sqcup X_\ell$ be a system of imprimitivity. This is the partition induced by a non trivial equivalence relation $\sim$ which is $G$-invariant (i.e. $x\sim y$ implies $gx\sim gy$).
Since $G$ acts transitively, we recall that $|X_i|=|X_1|$ for all $i\in\{1,\dots \ell\}$.
We argue by contradiction, by assuming $1<|X_1|<n$.
Consider now the cycle $\sigma$ of order $r$ and take $X_j$ which intersects the support of $\sigma$ (i.e. $\sigma$ acts non trivially on $X_j$).
Consider the orbit of $X_j$ via $\sigma$:
\[X_j,\sigma(X_j),\dots, \sigma^{v-1}(X_j),\]
where $v$ is the orbit of $X_j$ via $\sigma$.
We have that $v$ necessarily divides $r$. Then both $v=1$ and $v=r$ are impossible.
\end{proof}


\section{A characterization of universal polynomials}\label{sec:characterizationofuni}

We are now ready to prove the main result.

\begin{theorem}\label{monodromy}
Let $f\in \vF_q[X]$. Suppose that $n=\deg(f)\geq 8$, then $f$ is universal if and only if $A_f=G_f=\mathcal S_{n}$.
\end{theorem}
\begin{proof}
First, let us assume that $f$ is $d$-universal for some positive integer $d$. Consider first 
\[A_f'=\Gal(\vF_{q^d}M_f:\vF_{q^d}(t))\leq \mathcal S_n.\] 
Let $x$ be any zero of $f(X)-t$ over $\overline{\vF_q(t)}$. From now on, we will look at $A_f'$ as a subgroup of the permutation group of the roots of $f(X)-t$ (or equivalently of the set $H=\Hom_{\vF_{q^d}(t)}(\vF_{q^d}(x),\vF_{q^d} M_f)$). Our first purpose is indeed to show that $A_f'=\cS_n$.

Let $r$ be a prime in $\{\lfloor\frac{n}{2}\rfloor+1,\dots,n-3\}$. Such prime always exists by Bertrand Postulate (also known as Chebyshev's Theorem).
Fix now $t_0\in\vF_{q^d}$ in such a way that $f(X)-t_0$ has an irreducible factor $h(X)$ of degree $r$ (over $\vF_{q^d}[X]$). This implies  immediately  that the ramification at $t_0$ is one, as $h(X)^e$ would have degree larger than $n$ for any $e>1$.
We claim that there exists $\gamma\in A_f'$ which is a cycle of order $r$. 
Let $P$ be the place corresponding to $t_0$, $Q$ be the place of 
$\vF_{q^d}(x)$ corresponding to the irreducible factor of degree $r$  lying over $P$, and $R$ be a place of $\vF_{q^d}M_f$ lying over $Q$.
Let $g\in D(R|P)$ be such that its image in $\Gal(\cO_R/R:\cO_P/P)$ under the natural reduction modulo $R$ is the Frobenius automorphism. The order of $g$ is then divisible by 
$r$, since an orbit of $g$ acting on $H=\Hom_{\vF_{q^d}(t)}(\vF_{q^d}(x),\vF_{q^d} M_f)$ has size $r$ (by the natural correspondence given by Theorem \ref{orbits}). As $r$ is prime, the only chance is that $g$ has a cycle of order $r$ in its decomposition in disjoint cycles. Now, as $r>n/2$, a certain power of $g$ will be a cycle of order $r$: this is our element $\gamma$.

Let us now summarize the properties of $A_f'$ given by the $d$-universality:
\begin{enumerate}
\item It contains a cycle of order $n/2<r<n-2$ (by the previous argument and a direct application of Theorem \ref{orbits}).
\item Since $f(X)-t_1$ is irreducible for some $t_1$, we get that  $A_f'$ contains a cycle of order $n$ by a direct application of Theorem \ref{orbits}. 
\item Analogously, it contains a cycle of order $n-1$.
\end{enumerate}

(1)+Lemma \ref{primitivegroup} implies that $A_f'$ is primitive, therefore, (1)+(2) implies that $A_f'$ contains the alternating group thanks to a theorem of Jordan \cite[Theorem 13.9]{wielandt2014finite}.
Then (2)+(3) implies that $A_f'$ is not the alternating group.
It follows that $A_f'= \mathcal S_n$.
Let us now show that $A_f=A_f'$.
Recall that $\tilde k$ is the constant field of $M_f$. Let $k'=\vF_{q^d}$ and $\tilde{k}'=\tilde k\cap k'$. By Lemma \ref{compositumfinitefield}
\[\mathcal S_n=A'_f=\Gal(k'M_f:k'(t))=\Gal(M_f:\tilde{k}'(t)).\]
Now, by observing $\Gal(M_f:\tilde{k}'(t))\leq \Gal(M_f:\vF_q(t))=A_f\leq \mathcal S_n$ we conclude $A_f'=A_f$.

We have now to show that the field of constants of $M_f$ is indeed $\vF_q$.
The only other possibility is that the field of constants is $\tilde k=\vF_{q^2}$ as for $n\geq 5$,  $\mathcal S_n$ has no normal subgroups other than the alternating group $\mathcal A_n$. 
The reader should notice that if $d$ is even then $\tilde k'=\tilde k=\vF_{q^2}$, therefore we are done by the fact that $G_f=\Gal(M_f:\tilde{k}'(t))=\Gal(M_f:\tilde{k}(t)) =\cS_n$.
Thus, we restrict to the case $d$ odd. Let us argue by contradiction by supposing $k'\tilde k=\vF_{q^{2d}}$.
Suppose that $n=\deg(f)$ is odd, and let $t_1\in \vF_{q^d}$ for which $f(x)-t_1$ is irreducible of degree $n$. Let us denote by $P_1$ the place corresponding to $t_1$ in $\vF_{q^d}(t)$, $Q\subset \vF_{q^{2d}}(x)$ be the place over $P_1$ corresponding to the irreducible polynomial $f(x)-t_1$, and $R$ a place of $\vF_{q^d}M_f$ lying over $Q$. Since $Q$ is unique and unramified, then $R$ is unramified. Therefore, $D(R|P_1)$ is cyclic and it has exactly one orbit of order $n$ corresponding to $Q$ under the bijection given by Theorem \ref{orbits}. It follows that any generator of $D(R|P_1)$ is a cycle of order $n$, so $D(R|P_1)$ has order $n$. On the other hand, the order of $D(R|P_1)$ is also $f(R|P_1)$, which is divisible by  $[\vF_{q^{2d}}:\vF_{q^{d}}]=2$ thus we have a contradiction.

If $n$ is even, then take $t_2$ for which $f(x)-t_2$ has an irreducible factor $h(X)$ of degree $n-1$ (and therefore also a factor of degree $1$). Let $P_2$ be the place corresponding to $t_2$ and $Q_1$, $Q_2$ be the places of $\vF_{q^{2d}}(x)$ corresponding respectively to $h(X)$ and to the factor of degree one of $f(x)-t_2$. Let $R$ be a place of $\vF_{q^d}M_f$ lying over $P_2$. Since $Q_1$ and $Q_2$ are the unique places of $\vF_{q^{2d}}(x)$ lying over $P_2$ and they are both unramified, then any place $R$ lying above $P_2$ is unramified.
Arguing similarly as before, we get that $D(R|P_2)$ is cyclic and it has a cycle of order $n-1$, therefore $f(R|P_2)=|D(R|P_2)|=n-1$. On the other hand, since the size of the decomposition group is divided by $[\vF_{q^{2d}}:\vF_{q^{d}}]=2$, we get the contradiction we wanted.

This shows that the constant field of $\vF_{q^d} M_f$ is $\vF_{q^d}$. On the other hand, the field of constants of $\vF_{q^d}M_f$ is $\tilde k \vF_{q^d}$: as $d$ is odd, this forces $\tilde k=\vF_q$ (as the only other chance was $k'=\vF_{q^{2d}}$).

%
%

Let us prove the other implication. Suppose that $G_f=A_f=\cS_n$ and fix $\ell\in\{1,\dots,n\}$. Let now $\gamma$ be a cycle of $G_f$ of order $\ell$ and let $\Gamma$ be its conjugacy class. In the notation of Theorem \ref{handlecheb}, for any $d\in \vN$
we have that 
\[|\{P\in S_{\vF_{q^d}(t)}| \deg_{\vF_{q^d}}(P)=1, (P,M:K)=\Gamma\}|=\frac{|\Gamma|}{|G_f|}q^{d}+O\left(q^{d/2}\right),\]
where the implied constant is independent of $d$ and $q$.
This shows immediately that, when $d$ is large enough, there is an unramified place $P$ of degree $1$ in $\vF_{q^d}(t)$ (corresponding to an element $t_0\in \vF_{q^d}$) for which $\gamma$ is the Frobenius (for some place of $\vF_{q^d}M_f$ lying over $P$). As $\gamma$ is a cycle of order $\ell$, by applying Theorem \ref{orbits} we get that 
$f(X)-t_0$ has a factor of degree $\ell$ in $\vF_{q^d}[X]$.
\end{proof}

\begin{remark}
The philosophy behind the proof of the first implication of Theorem \ref{monodromy} can be applied to prove similar statements, so we highlight the two most important steps here. 
The first step is to use the property we are interested in to obtain a bunch of group theoretical conditions via Theorem \ref{orbits} (conditions 1,2,3 in the proof of Theorem \ref{monodromy}). Then, once the candidate arithmetic monodromy group is described, we have to understand how the property we require  from $f$ (in this case universality), combined with the group theoretical properties we found, affect the possible field of constants of $M_f$ (in our case we could prove that $k_f$ is trivial).  
The critical advantage of the method is that one can use powerful group theoretical machinery to obtain complete characterization of monodromy groups. Moreover, via Chebotarev Density Theorem, the monodromy groups capture all the splitting statistics of the map $f$ as long as the base field is large enough compared with the degree of $f$.
\end{remark}

\begin{remark}
The reader should notice that the second implication i.e. the last paragraph of the above proof can also be deduced by \cite[Theorem 1]{cohen1970distribution}.
\end{remark}

\begin{corollary}\label{strong_cor}
Suppose that $f$ is $\overline d$-universal for some $\overline d$, then there exists $d_0$ for which $f$ is $d$-universal for every $d>d_0$.
\end{corollary}
\begin{proof}
Suppose that $f$ is $\overline d$-universal, then $G_f=A_f=S_n$. By the same argument as in the proof the second implication of Theorem \ref{monodromy}, it follows that when $d$ is large enough the number of $t_0\in \vF_{q^d}$ for which $f(x)$ has an irreducible factor of degree $\ell$ can be estimated with $\frac{|\Gamma|}{|G_f|}q^d$ for $\Gamma$ the conjugacy class of an element having a cycle of order $\ell$ in its decomposition in disjoint cycles (one can actually directly select an element which "is" a cycle of order $\ell$).
\end{proof}

\begin{corollary}
A universal polynomial $f$ of degree greater than or equal to $8$ is indecomposable, i.e. it cannot be written as composition of lower degree polynomials
\end{corollary}
\begin{proof}
By Theorem \ref{monodromy}, it is enough to observe that $\mathcal S_n$ acts primitively on the roots of $f$. This forces the polynomial to be indecomposable (see for example \cite[Section 2.3]{muller1995primitive}).
\end{proof}


\section{Universality for $X^q+X^2-t$}\label{sec:Xq+X2-t}
In this section let us specialize to the polynomial $f=X^q+X^2$, as this is the one suggested for the function field sieve \cite{Barbulescu2014} and experimentally is believed to be  $2$-universal, see last paragraph in \cite[Section 5]{Barbulescu2014}.
For this section we will restrict to $q$ odd.
Let us recall a result due to Turnwald \cite{turnwald1995schur}.
\begin{theorem}\label{galoisgroupturnwald}
Let $k$ be a field of characteristic different from $2$ and 
$g\in k[X]$. Suppose that the derivative $g'$ of $g$ has at least a simple root and
 for any pair of roots $\alpha,\beta$ of $g'$ over $\overline{k}$ we have that
  $g(\alpha)\neq g(\beta)$. 
  In addition suppose that $\Chr(k)
  \nmid 
  \deg(g)$.
Then the Galois Group of $g-t$ over $k(t)$ is $\mathcal S_{\deg(g)}$.
\end{theorem}

With this tool in hand, we are able to compute the  arithmetic and the geometric monodromy group of $X^q+X^2$.

\begin{proposition}\label{galoisgroupmypol}
Let $\Chr(\vF_q)\neq 2$ and $f=X^q+X^2\in \vF_q[X]$.
The Galois Group $A_f$ of $f-t\in \vF_q(t)[X]$ over $\vF_q(t)$ is 
$\mathcal S_q$. Moreover $G_f=A_f$.
\end{proposition}
\begin{proof}
Clearly, Theorem \ref{galoisgroupturnwald} does not apply to the polynomial above as its degree is divisible by the characteristic of the field.
Let us consider the extension $\vF_q(x):\vF_q(t)$ where $x$ is a root of $f-t$, and then verifies $f(x)=x^q+x^2=t$. Let $\{y_1,\dots,y_{q-1}\}$ be the set of roots of $X^{q-1}+X+2x\in\vF_q(x)[X]$. They are all distinct, as the polynomial is separable. It is easy to see that $x+y_i$ is a root of $f-t$ for any $i\in \{1,\dots q-1\}$.
Therefore, the splitting field $M_f$ of $f-t$ over $\vF_q(t)$ is exactly $\vF_q(x,y_1,\dots,y_{q-1})$. 
Let us now consider  $B=\Gal(M_f:\vF_q(x))$ which is a subgroup of $A_f=\Gal(M_f:\vF_q(t))$. The Galois Group $B$ is the same as the Galois group of the polynomial $\frac{X^{q-1}+X}{-2}-x$ over $\vF_q(x)$, for which Turnwald theorem applies with base field $\vF_q(x)$ since
\begin{itemize}
\item $\Chr(\vF_q)\neq 2$
\item The roots of $X^{q-2}-1$ are $\xi^i$ for $\xi$ a primitive $(q-2)$-root of unity and $i\in\{0,\dots,q-3\}$.
\item  $\frac{\xi^{i(q-1)}+\xi^i}{-2}=\frac{\xi^{j(q-1)}+\xi^j}{-2}$ implies $\xi^{i}=\xi^{j}$ but then $i=j$.
\end{itemize}
We are now sure that the Galois Group of $\frac{X^{q-1}+X}{-2}-x$  is $B=\cS_{q-1}$. Observe that $B\leq A_f$ and $A_f$ acts transitively on the set of roots $\{x,x+y_1,x+y_2,\dots, x+y_{q-1}\}$ and the stabilizer of $x$ contains $B$. By the orbit-stabilizer theorem we have that
\[q=\frac{|A_f|}{|\St_{A_f}(x)|}\leq \frac{|A_f|}{|B|}=\frac{|A_f|}{(q-1)!}\]
Therefore $|A_f|\geq q!$ but also $|A_f|\leq q!$ as $A_f$ is a subgroup of $\cS_q$, so $A_f=\cS_q$. 
We have now to show that $G_f=A_f$. Suppose that the constant field of $M_f$ is $\tilde k$ and notice that all the arguments above apply again by replacing $\vF_q$ with $\tilde k$. Hence this immediately shows $G_f=\mathcal S_q$.
\end{proof}

\begin{corollary}\label{finalcorollary}
There exists $d_0\in\vN$ such that $X^q+X^2$ is $d$-universal for any  $d>d_0$.
\end{corollary}
\begin{proof}
By individually checking the cases $q<8$ we can assume $q\geq 8$.
By the previous result we have that Theorem \ref{monodromy} applies, therefore it also applies Corollary \ref{strong_cor}, which is exactly the claim.
\end{proof}
The reader should notice now that the first occurence of $d_1$ for which $X^q+X^2$ is $d_1$ universal might be strictly less than $d_0$.
What would be ideal to show, is that  $d_0$ is indeed ``small" enough (conjecturally it is $2$), on the other hand the above corollary at least shows that such $d$ exists.

\section{Constructing $d$-universal polynomials in odd characteristic}

The combination of Theorem \ref{galoisgroupturnwald} and Theorem \ref{monodromy} gives a deterministic easy way to construct polynomials which are likely to build up any extension between the base field and the degree of the polynomial satisfying Conjecture \ref{strongerconj}.
We give a class of examples in the next subsection. For the rest of this section, $q$ will be an odd prime power.
\subsection{Universality for $X^{q+j}-jX$}\label{xqjjx}

In this subsection we show a large class of polynomials which can be shown to be universal. In addition such polynomials appear to be always $d$-universal for a small $d$. 
\begin{proposition}
Let $q$ be an odd prime power and $\vF_p$ be its prime subfield. Let $j\in \vN\setminus \{0,1,pk\}_{k\in \vN}$. The polynomial $f=X^{q+j}-jX\in \vF_q[X]$ is universal. 
\end{proposition}
\begin{proof}
We would like to verify the conditions of Theorem \ref{galoisgroupturnwald} for the geometric monodromy group of $f$, then it will follow that also the arithmetic monodromy group of $f$ is the symmetric group, for which Theorem \ref{monodromy} now applies, showing the universality of $f$.

The derivative of $f$ is $f'=j X^{q+j-1}-j$. Since $j$ is different from $1$, then $f'$ has all single roots in $\overline \vF_q$. Now, any root of $f$ has the form $\xi^u$, where $\xi$ is a fixed primitive $q+j-1$ root of unity, and $u$ is an integer in $\{0,\dots, q+j-2\}$. 

It is now enough to observe that  $f(\xi^u)= \xi^u -j \xi^u=(1-j)\xi^u  \neq (1-j)\xi^v=f(\xi^v)$ for $u\neq v \mod q+j-1$.

The conditions of Theorem \ref{galoisgroupturnwald} are now verified and then Theorem \ref{monodromy} applies, leading to the claim.
\end{proof}

\begin{remark}
The experiments show that this class of polynomials actually verifies a stronger property, i.e. each of them seems to be $d$-universal for $d=j+1$. 
In particular, for $j=2$, the polynomial $X^{q+2}-2X$ is $3$-universal for any prime $q$ less than\footnote{The computations were performed in SAGE and the code is available upon request} $401$ therefore building up suitable extensions of size up to $401^{401}$.
\end{remark}

\section*{Acknowledgements}
The author is grateful to Michael Zieve for many interesting discussions and especially for introducing him to the version of Chebotarev Density Theorem used in this paper. The author also wants to thank Swiss National Science Foundation grant number 171248.

\bibliographystyle{plainnat}
\bibliography{biblio}{}

\end{document}